\documentclass[12pt]{article}

\usepackage{amsfonts}
\usepackage{latexsym}

\usepackage{amsmath}
\usepackage{amssymb}
\usepackage{amsthm}
\usepackage{amscd}
\usepackage{mathrsfs}

\usepackage{float}
\theoremstyle{plain}
\newtheorem{theorem}{Theorem}

\newtheorem{lemma}[theorem]{Lemma}

\theoremstyle{definition}

\theoremstyle{remark}

\newcommand{\seqnum}[1]{\underline{#1}}
\begin{document}

\title{\Large\bf On sums of the small divisors of a natural number}

\author{\it Douglas E. Iannucci\\ \it  University of the Virgin Islands\\ \it 2 John Brewers Bay \\ \it St. Thomas VI 00802\\ \tt diannuc@uvi.edu}\date{}

\maketitle

\begin{abstract}We consider the positive divisors of a natural number that do not exceed its square root, to which we refer as the {\it small divisors\/} of the natural number. We determine the asymptotic behavior of the arithmetic function that adds the small divisors of a natural number, and we consider its Dirichlet generating series. 
\end{abstract}

\section{Introduction}\label{intro} By the {\it small divisors\/} of a natural number~$n$, we mean the set of integers
$$\{d: d\mid n, 1\le d\le\sqrt{n}\}.$$ 
The phrase ``small divisors,'' as defined here, is not to be confused with classical small divisors problems of mathematical physics (see, e.g., Yoccoz~\cite{yaccoz}). Aside from an earlier paper by the author~\cite{ian}, our definition of this phrase seems absent from the literature.
Define the arithmetic function~$a$ by
$$a(n)=\sum_{d\mid n \atop d\le\sqrt{n}}d,$$
the sum taken over natural numbers. Thus $a(n)$ adds the small divisors of~$n$. The sequence $a(n)$ appears as sequence \seqnum{A066839} in the {\it OEIS}~\cite{oeis}. We have the trivial bound,
$$a(n)\le\sum_{k=1}^{[\sqrt{n}]}k=
\frac12\left[\sqrt{n}\right]\left(\left[\sqrt{n}\right]+1\right)\le
\frac12\left(n+\sqrt{n}\right)\le n.$$
A.~W.~Walker has pointed out that
$$a(n)=\sum_{{d\mid n}\atop{d\le\sqrt{n}}}d\le\sqrt{n}\sum_{d\mid n}1
=\sqrt{n}\,\tau(n),$$
where $\tau(n)$ denotes the sum of {\it all\/} the positive divisors of~$n$. As
$$\lim_{n\to\infty}\frac{\tau(n)}{n^{\delta}}=0$$
for all $\delta>0$ (see Apostol~\cite[Theorem~13.12]{apostol}), it follows that
$$\lim_{n\to\infty}\frac{a(n)}{n^{\frac12+\delta}}=0$$
for all $\delta>0$. Thus, it seems that~$a(n)$ compares with~$\sqrt{n}$. In \S~\ref{asymp}, we in fact prove that $a(n)$ has average order~$\sqrt{n}$.  In \S~\ref{dirichlet} we obtain some properties of the Dirichlet generating series for~$a(n)$. 

We observe here that the function $a(n)$ is not multiplicative. It is, however, supermultiplicative:
\begin{lemma}\label{supermult}
If~$m$ and~$n$ are relatively prime natural numbers, then $a(mn)\ge a(m)a(n)$.
\end{lemma}
\begin{proof}
Suppose~$d_1$ and~$d_2$ are small divisors of~$m$, and~$d_1'$ and~$d_2'$ are small divisors of~$n$. Since $\gcd(m,n)=1$, we have $d_1d_1'=d_2d_2'$ if and only if $d_1=d_2$ and $d_1'=d_2'$. Therefore the product
\begin{equation}\label{smsum}
a(m)a(n)=\Big(\sum_{d\mid m \atop d\le\sqrt{m}}d\Big)\Big(\sum_{d'\mid n \atop d'\le\sqrt{n}}d'\Big)
\end{equation}
gives a sum, all of whose addends are distinct small divisors of~$mn$. Therefore $a(m)a(n)\le a(mn)$.
\end{proof}
Note that $a(24)=10$ and $a(36)=16$. Yet, $26\cdot36=864$ and $a(864)=130<160$. Hence $a(n)$ is not completely supermultiplicative.

\section{Asymptotic behavior of $\boldsymbol{a(n)}$}\label{asymp}
Two functions $f(x)$ and $g(x)$ are said to be {\it asymptotic\/} when
$$\lim_{x\to\infty}\frac{f(x)}{g(x)}=1,$$
and we denote this by $f(x)\sim g(x)$. We shall use the notation of Bachmann and Landau, viz., 
$$f(x)=O\left(g(x)\right),$$
whenever $|f(x)|\le C|g(x)|$ as $x\to\infty$ for some positive constant~$C$ independent of~$x$. 

If $f(n)$ and $g(n)$ are arithmetic functions, we say $f(n)$ is of {\it average order\/} $g(n)$ whenever
$$\sum_{k=1}^n f(k)\sim \sum_{k=1}^n g(k)$$
(e.g., see Hardy and Wright~\cite[\S\ 18.2]{HW}). 
\begin{theorem}\label{avorder}
The function $a(n)$ is of average order $\sqrt{n}$. More precisely,
\begin{equation}\label{avordereq}
\sum_{k=1}^n a(k)=\frac23 n\sqrt{n} + O(n\ln n).
\end{equation}
\end{theorem}
\begin{proof}
Equation~\eqref{avordereq} proves the theorem, for, by elementary calculus,
$$\sum_{k=1}^n\sqrt{k}=\frac23 n\sqrt{n} + O(\sqrt{n}).$$
Note that
\begin{equation}\label{splitab}
\sum_{k=1}^n a(k)=\sum_{k=1}^n\sum_{d\mid k \atop d\le\sqrt{k}}d=\sum_{(x,y)\in A}y+\sum_{(x,y)\in B}y,
\end{equation}
where the ordered pairs $(x,y)$ range over all lattice points (that is, where $x$, $y\in\mathbb{Z}$) of two regions, $A\subset\mathbb{R}^2$ and~$B\subset\mathbb{R}^2$,  which are defined as follows,
\begin{align*}
A&=\{(x,y) : 0< y\le x\le\sqrt{n}\},\\
B&=\{(x,y) : \sqrt{n}<x,\;0< y\le n/x\,\},\end{align*}
and which are depicted in Figure~\ref{hwpf}. 

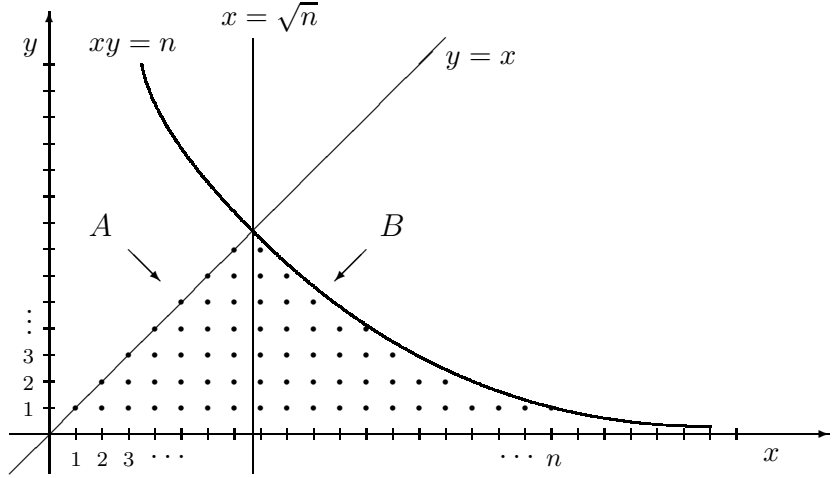
\begin{figure}
\caption{Lattice points in regions $A$ and $B$}
\label{hwpf}
\begin{center}
\begin{picture}(325,190)(-25,-25)
\put(-15,0){\vector(1,0){310}}\put(270,-10){\small$x$}
\multiput(10,-2)(10,0){26}{\line(0,1){4}}
\put(8,-12){\scriptsize{1}}\put(18,-12){\scriptsize{2}} 
\put(28,-12){\scriptsize{3}}\put(38,-9){\dots} 
\put(170,-9){\dots}\put(188,-12){\footnotesize{$n$}} 
\put(0,-15){\vector(0,1){175}}\put(-10,145){\small$y$}
\multiput(-2,10)(0,10){14}{\line(1,0){4}}
\put(-10,7){\scriptsize{1}}\put(-10,17){\scriptsize{2}} 
\put(-10,27){\scriptsize{3}}\put(-9,38){$\vdots$} 
\put(77,-15){\line(0,1){165}}\put(65,155){\small$x=\sqrt{n}$}
\put(-15,-15){\line(1,1){165}}\put(150,140){\small$y=x$}
\linethickness{0.7pt}\qbezier(77,77)(150,3)(250,3)
\linethickness{0.7pt}\qbezier(35,140)(40,114)(77,77)
\put(15,145){\small$xy=n$}
\multiput(10,10)(10,0){19}{\circle*{2}}
\multiput(20,20)(10,0){14}{\circle*{2}}
\multiput(30,30)(10,0){11}{\circle*{2}}
\multiput(40,40)(10,0){9}{\circle*{2}}
\multiput(50,50)(10,0){6}{\circle*{2}}
\multiput(60,60)(10,0){4}{\circle*{2}}
\multiput(70,70)(10,0){2}{\circle*{2}}
\put(15,75){$A$}\put(30,70){\vector(1,-1){12}}
\put(125,75){$B$}\put(120,70){\vector(-1,-1){12}}
\end{picture}
\end{center}
\end{figure}

Then,
$$\sum_{(x,y)\in A}y=\sum_{x=1}^{[\sqrt{n}]}\sum_{y=1}^{x}y
=\frac12\sum_{x=1}^{[\sqrt{n}]}x(x+1)
=\frac12\sum_{x=1}^{[\sqrt{n}]}x^2+\frac12\sum_{x=1}^{[\sqrt{n}]}x.$$
The first of these two sums yields
$$\frac12\sum_{x=1}^{[\sqrt{n}]}x^2
=\frac{[\sqrt{n}]([\sqrt{n}]+1)(2[\sqrt{n}]+1)}{12}
=\frac{(\sqrt{n}+O(1))^2(2\sqrt{n}+O(1))}{12},$$
while the second sum yields
$$\frac12\sum_{x=1}^{[\sqrt{n}]}x=\frac{[\sqrt{n}]([\sqrt{n}]+1)}4=\frac{(\sqrt{n}+O(1))^2}4,$$
where we have applied $[x]=x+O(1)$ for all real~$x$. Therefore
\begin{align*}
\sum_{(x,y)\in A}y&=\frac{(\sqrt{n}+O(1))^2(2\sqrt{n}+O(1))}{12}+\frac{(\sqrt{n}+O(1))^2}4\\
&=\frac1{12}(\sqrt{n}+O(1))^2(2\sqrt{n}+O(1)),\end{align*}
which yields
\begin{equation}\label{regiona}
\sum_{(x,y)\in A}y=\frac{n\sqrt{n}}6+O(n).
\end{equation}
Next, we have
\begin{align*}
\sum_{(x,y)\in B}y&=\sum_{x=[\sqrt{n}]+1}^{n}\sum_{y=1}^{[n/x]}y\\
&=\frac12\sum_{x=[\sqrt{n}]+1}^{n}\left[\frac{n}{x}\right]\left(\left[\frac{n}{x}\right]+1\right)\\
&=\frac12\sum_{x=[\sqrt{n}]+1}^{n}\left(\frac{n}{x}+O(1)\right)^2,
\end{align*}
yielding
$$\sum_{(x,y)\in B}y=\frac{n^2}2\sum_{x=[\sqrt{n}]+1}^{n}\frac1{x^2}+O(n\ln n),$$
which follows because $\sum_{x=a}^{b}\frac1{x}=O(\ln b)$ for all $a$, $b\in\mathbb{N}$, $a<b$. By elementary calculus,
$$\sum_{x=[\sqrt{n}]+1}^{n}\frac1{x^2}=\frac1{\sqrt{n}}+O\left(\frac1{n}\right),$$
hence
\begin{equation}\label{regionb}
\sum_{(x,y)\in B}y=\frac{n\sqrt{n}}2+O(n\ln n).\end{equation}
As~\eqref{avordereq} follows immediately from~\eqref{splitab}, \eqref{regiona}, and~\eqref{regionb}, the proof is complete.
\end{proof}
\vskip 12pt

It is thus natural to consider the behavior of the sequence $a(n)/\sqrt{n}$. Perhaps unsurprisingly, this behavior is irregular. For instance, it is clear that 
$$\lim\inf \frac{a(n)}{\sqrt{n}}=0,$$
as $a(p)=1$ for all primes~$p$. On the other hand, it is easy to see that
$$\lim\sup \frac{a(n)}{\sqrt{n}}=\infty.$$
For, we need only consider the sequence $s_n=p_1^2p_2^2\cdots p_n^2$, where the primes are enumerated as $p_1=2$, $p_2=3$, and so on. Every number of the form $p_1^{\epsilon_1}p_2^{\epsilon_2}\cdots p_n^{\epsilon_n}$, where $\epsilon_k=0$ or~1 for $1\le k\le n$, is a small divisor of~$s_n$. Therefore
$$a(s_n)\ge\sum p_1^{\epsilon_1}p_2^{\epsilon_2}\cdots p_n^{\epsilon_n}=\prod_{k=1}^{n}(p_k+1),$$
where the sum ranges over all $n$-tuples $(\epsilon_1, \epsilon_2,\dots, \epsilon_n)$ where $\epsilon_k=0$ or~1 for $1\le k\le n$. Hence
$$\lim_{n\to\infty}\frac{a(s_n)}{\sqrt{s_n}}\ge\lim_{n\to\infty}\prod_{k=1}^n\left(1+\frac1{p_k}\right)=\infty.$$

The average order of~$a(n)$ is interesting when compared to that of the function $\sigma(n)$, which adds {\it all\/} the positive divisors of~$n$,
$$\sigma(n)=\sum_{d\mid n}d.$$
The sequence $\sigma(n)$ appears as \seqnum{A000203} in the {\it OEIS\/}. The average order of~$\sigma(n)$ is $\frac{\pi^2}6n$ (see Hardy and Wright~\cite[\S\ 18.3,\ Theorem\ 324]{HW}), i.e., we have a nonunit multiple of~$n$ ($\frac{\pi^2}6\approx1.645$), as compared to Theorem~\ref{avorder} (merely $\sqrt{n}$ for the average order of $a(n)$).

\section{The Dirichlet series of ${\boldsymbol{a(n)}}$}\label{dirichlet}
An arithmetic function $f(n)$ is said to have a {\it Dirichlet generating series\/}, defined by
$$L(s,f)=\sum_{n=1}^{\infty}\frac{f(n)}{n^s}.$$
Following Riemann, we let~$s$ be a complex variable and write
$$s=\sigma+it,$$
where~$\sigma$ and~$t$ are real; in particular $\sigma=\text{Re} (s)$. Hence $|n^s|=n^{\sigma}$, therefore
$$\sum_{n=1}^{\infty}\left|\frac{a(n)}{n^s}\right|=\sum_{n=1}^{\infty}\frac{a(n)}{n^{\sigma}}.$$
Since $a(n)\ge1$ for all~$n\in\mathbb{N}$, it follows that 
\begin{equation}\label{exampleseries}
\sum_{n=1}^{\infty}\frac{a(n)}{n^{\sigma}}\end{equation}
diverges for all $\sigma\le1$; similarly, as $a(n)\le n$ for all~$n$, it follows that the series~\eqref{exampleseries} converges for all $\sigma>2$ (see Apostol~\cite[Theorem~11.8]{apostol}).
Therefore, there exists $\alpha\in\mathbb{R}$, $1<\alpha\le2$, such that the Dirichlet series $L(s,a)$ converges on the half-plane $\sigma>\alpha$, but does not converge on the half-plane $\sigma<\alpha$. Here, $\alpha$ is called the {\it abscissa of convergence\/} of $L(s,a)$ (see Apostol~\cite[Theorem~11.9]{apostol}). 

Recall that the Dirichlet series $L(s,1)$ is the Riemann zeta function when $\sigma>1$, and that $L(s,1)$ has $\alpha=1$ as its abscissa of convergence. We write $\zeta(s)=L(s,1)$. 

Thus it follows that $L(s,\sqrt{n})=\zeta\left(s-\frac12\right)$, and has as its abscissa of convergence $\alpha=3/2$. Therefore, in light of Theorem~\ref{avorder}, we expect the same abscissa of convergence for~$a(n)$. 

\begin{theorem}\label{abscissa}
The abscissa of convergence for the Dirichet series $L(s,a)$ is given by $\alpha=3/2$.
\end{theorem}
\begin{proof}
We need only show that the series~\eqref{exampleseries} diverges at~$\sigma=3/2$, and converges for~$3/2<\sigma<2$ (for, $L(\sigma,a)$ decreases as $\sigma\in\mathbb{R}$ increases). 

First we consider 
\begin{align*}
\sum_{k=1}^n\frac{a(k)}{k^{3/2}}&=\sum_{k=1}^n\frac1{k^{3/2}}\sum_{d\mid k \atop d\le\sqrt{k}}d\\
&=\sum_{k=1}^n\sum_{d\mid k \atop d\le\sqrt{k}}\frac1{(k/d)^{3/2}}\cdot\frac1{d^{1/2}}\\
&=\sum_{(x,y)\in A}\frac1{x^{3/2}}\cdot\frac1{y^{1/2}}+\sum_{(x,y)\in B}\frac1{x^{3/2}}\cdot\frac1{y^{1/2}},\end{align*}
where~$A$ and~$B$ are defined as in the proof of Theorem~\ref{avorder} (see Figure~\ref{hwpf}). 

By elementary calculus,
$$\sum_{y=1}^{x}\frac1{y^{1/2}}\ge\int_1^x\frac{dy}{y^{1/2}}= 2x^{1/2}-2,$$
hence
\begin{gather*}
\sum_{(x,y)\in A}\frac1{x^{3/2}}\cdot\frac1{y^{1/2}}=\sum_{x=1}^{[\sqrt{n}]}\frac1{x^{3/2}}\sum_{y=1}^{x}\frac1{y^{1/2}}
\ge2\sum_{x=1}^{[\sqrt{n}]}\frac1{x}-2\sum_{x=1}^{[\sqrt{n}]}\frac1{x^{3/2}}\\
\ge2\log\left[\sqrt{n}\right]-2\zeta(3/2),\end{gather*}
where we applied $\sum_{x=1}^m1/x\ge\log{m}$ for all $m\in\mathbb{N}$. Clearly,
$$\sum_{(x,y)\in B}\frac1{x^{3/2}}\cdot\frac1{y^{1/2}}\ge0,$$
hence
$$\sum_{k=1}^n\frac{a(k)}{k^{3/2}}=\sum_{(x,y)\in A}\frac1{x^{3/2}}\cdot\frac1{y^{1/2}}+\sum_{(x,y)\in B}\frac1{x^{3/2}}\cdot\frac1{y^{1/2}}\ge2\log\left[\sqrt{n}\right]-2\zeta(3/2),$$
which diverges to infinity as $n\to\infty$; thus the series~\eqref{exampleseries} diverges at~$\sigma=3/2$.

Next we consider $\frac32<\sigma<2$. We remark that for $M\in\mathbb{N}$ we have
\begin{equation}\label{aux01}
\sum_{y=1}^M\frac1{y^{\sigma-1}}\le1+\int_1^M\frac{dy}{y^{\sigma-1}}\le\frac{M^{2-\sigma}}{2-\sigma}.\end{equation}
Here,
\begin{align*}
\sum_{k=1}^n\frac{a(k)}{k^{\sigma}}&=\sum_{k=1}^n\frac1{k^{\sigma}}\sum_{d\mid k\atop d\le\sqrt{k}}d\\
&=\sum_{k=1}^n\sum_{d\mid k\atop d\le\sqrt{k}}\frac1{(k/d)^{\sigma}}\cdot\frac1{d^{\sigma-1}}\\
&=\sum_{(x,y)\in A}\frac1{x^{\sigma}}\cdot\frac1{y^{\sigma-1}}+\sum_{(x,y)\in B}\frac1{x^{\sigma}}\cdot\frac1{y^{\sigma-1}}.\end{align*}
Applying~\eqref{aux01}, we have both
\begin{align*}
\sum_{(x,y)\in A}\frac1{x^{\sigma}}\cdot\frac1{y^{\sigma-1}}&=\sum_{x=1}^{[\sqrt{n}]}\frac1{x^{\sigma}}\sum_{y=1}^x\frac1{y^{\sigma-1}}\\
&\le\sum_{x=1}^{[\sqrt{n}]}\frac1{x^{\sigma}}\cdot\frac{x^{2-\sigma}}{2-\sigma}\\
&=\frac1{2-\sigma}\sum_{x=1}^{[\sqrt{n}]}\frac1{x^{2(\sigma-1)}},\end{align*}
and
\begin{align*}
\sum_{(x,y)\in B}\frac1{x^{\sigma}}\cdot\frac1{y^{\sigma-1}}&=\sum_{x=[\sqrt{n}]+1}^n\frac1{x^{\sigma}}\sum_{y=1}^{[n/x]}\frac1{y^{\sigma-1}}\\
&\le\sum_{x=[\sqrt{n}]+1}^n\frac1{x^{\sigma}}\cdot\frac{[n/x]^{2-\sigma}}{2-\sigma}\\
&\le\frac{n^{2-\sigma}}{2-\sigma}\sum_{x=[\sqrt{n}]+1}^n\frac1{x^2},\end{align*}
hence
\begin{equation}\label{aux02}
\sum_{k=1}^n\frac{a(k)}{k^{\sigma}}\le
\frac1{2-\sigma}\sum_{x=1}^{[\sqrt{n}]}\frac1{x^{2(\sigma-1)}}+\frac{n^{2-\sigma}}{2-\sigma}\sum_{x=[\sqrt{n}]+1}^n\frac1{x^2}.\end{equation}
Clearly
\begin{equation}\label{aux03}
\sum_{x=1}^{[\sqrt{n}]}\frac1{x^{2(\sigma-1)}}\le\zeta(2(\sigma-1)).\end{equation}
We remark that
\begin{equation}\label{aux04}
\sum_{x=[\sqrt{n}]+1}^n\frac1{x^2}\le\frac1{\sqrt{n}}\;,\end{equation}
because
$$
\sum_{x=[\sqrt{n}]+1}^n\frac1{x^2}\le\frac1{([\sqrt{n}]+1)^2}+\int_{[\sqrt{n}]+1}^n\frac{dx}{x^2}\le\frac1{n}+\int_{\sqrt{n}}^n\frac{dx}{x^2}\,=\,\frac1{\sqrt{n}}\;.$$
Thus by~\eqref{aux02}, \eqref{aux03}, and~\eqref{aux04}, we have
\begin{align*}
\sum_{k=1}^n\frac{a(k)}{k^{\sigma}}&\le\frac{\zeta(2(\sigma-1))}{2-\sigma}
+\frac{n^{2-\sigma}}{2-\sigma}\cdot\frac1{\sqrt{n}}\\
&=\frac1{2-\sigma}\left(\zeta(2(\sigma-1))+\frac1{n^{\sigma-\frac32}}\right)\\
&\le\frac1{2-\sigma}\left(\zeta(2(\sigma-1))+1\right)\end{align*}
for all~$n\in\mathbb{N}$. Hence the series~\eqref{exampleseries} converges for all~$\sigma$ such that $\frac32<\sigma<2$.\end{proof}

We may define the arithmetic function $b(n)$ by $b(1)=1$, and, when~$n$ has unique prime factorization $n=p_1^{\beta_1}p_2^{\beta_2}\cdots p_k^{\beta_k}$, 
$$b(n)=a(p_1^{\beta_1})a(p_2^{\beta_2})\cdots a(p_k^{\beta_k}).$$
Thus $b(n)$ is multiplicative, and $b(n)\le a(n)$ for all $n\in\mathbb{N}$ by Lemma~\ref{supermult}. Note, then, that for all $\sigma>3/2$ we have $L(\sigma,b)\le L(\sigma,a)$. Furthermore, as~$b(n)$ is multiplicative, then $L(s,b)$ has an Euler product representation on its half-plane of convergence (see Apostol~\cite[Theorem~11.6]{apostol}), given by
\begin{align*}
L(s,b)&=\prod_{p}\left(1+\frac{b(p)}{p^{s}}+\frac{b(p^2)}{p^{2s}}+\frac{b(p^3)}{p^{3s}}+\cdots\right)
\\
&=\prod_p\left(1+\frac1{p^{s}}+\frac{p+1}{p^{2s}}+\frac{p+1}{p^{3s}}+\frac{p^2+p+1}{p^{4s}}+\frac{p^2+p+1}{p^{5s}}+\cdots\right)\\
&=\prod_p\left(1+\frac1{p^{2s-1}}+\frac1{p^{4s-2}}+\cdots\right)
\left(1+\frac1{p^s}+\frac1{p^{2s}}+\cdots\right)\\
&=\prod_p\left(1-\frac1{p^{2s-1}}\right)^{-1}\left(1-\frac1{p^{s}}\right)^{-1}\\
&=\zeta(2s-1)\zeta(s),\end{align*}
where the products are taken over all the primes~$p$. Note that the second line follows because
$$b(p^n)=a(p^n)=1+p+\cdots+p^{[n/2]}$$
for all primes~$p$ and integers~$n\ge0$. 

On the other hand, as $a(n)\le n$ for all natural numbers~$n$, we have for all $\sigma>2$,
$$L(\sigma,a)\le\sum_{n=1}^{\infty}\frac{n}{n^{\sigma}}=\zeta(\sigma-1).$$
Hence for all $\sigma>2$,
\begin{equation}\label{adir}
\zeta(2\sigma-1)\zeta(\sigma)\le L(\sigma,a)\le\zeta(\sigma-1).\end{equation}
In light of Theorem~\ref{avorder}, this is unsurprising, as, recalling $L(\sigma,\sqrt{n})=\zeta\left(\sigma-\frac12\right)$,
we see that the same bounds as in~\eqref{adir} hold for all~$\sigma>2$:
$$\zeta(2\sigma-1)\zeta(\sigma)\le L(\sigma,\sqrt{n})\le\zeta(\sigma-1).$$
The latter inequality is immediate, while the former follows because
$$\left(1-\frac1{p^{2\sigma-1}}\right)^{-1}\left(1-\frac1{p^{\sigma}}\right)^{-1}\le\left(1-\frac1{p^{\sigma-\frac12}}\right)^{-1}$$
for all primes~$p$ and all $\sigma>2$.

Note that 
$$\zeta(2s-1)=\sum_{n=1}^{\infty}\frac{n}{n^{2s}},$$
hence $\zeta(2s-1)=L(s,f)$, where
$$f(n)=\begin{cases}\sqrt{n},&\text{if $n$ is a square;}\\
0,&\text{otherwise.}\end{cases}$$
As $L(s,b)=\zeta(2s-1)\zeta(s)$, then (see Apostol~\cite[Theorem~11.5]{apostol})
$$b(n)=\sum_{d\mid n}f(d).$$
Thus $b(n)$ adds the square roots of the square divisors of~$n$. For example, $b(72)=1+2+3+6=12$; this compares to $a(72)=1+2+3+4+6+8=24$.

\bigskip
\hrule
\bigskip

\noindent 2010 {\it Mathematics Subject Classification}: 11A25.

\noindent \emph{Keywords:} small divisors, arithmetic function, average order, Dirichlet series.

\noindent (Concerned with sequences \seqnum{A066839} and \seqnum{A000203}.)
\end{document}